\documentclass[12pt]{article}
\usepackage[english]{babel}
\usepackage {amsmath, amssymb, amsthm}
\def\ep{\varepsilon}

\numberwithin{equation}{section}
\newtheorem{lem}{Lemma}[section]
\newtheorem{tm}{Theorem}[section]

\newtheorem{prop}{Proposition}[section]
\newtheorem{rem}{Remark}[section]
\newtheorem{dif}{Definition}[section]

\begin{document}

\centerline {\bf DEGENERATE SELF-SIMILAR MEASURES, SPECTRAL ASYMPTOTICS }
\centerline {\bf AND SMALL DEVIATIONS OF GAUSSIAN  PROCESSES}

\centerline {A.I.~Nazarov\footnote{Supported by grant RFBR No. 10-01-00154a
and by grant of scientific school No. 4210.2010.1}, I.A.~Sheipak\footnote{Supported by grants RFBR No. 10-01-00423a and RFBR No. 09-01-90408ukr\_f\_a}}

\section{Introduction}

The problem of small ball behavior for the norms of Gaussian processes is intensively studied in recent years. The simplest and most explored case is
that of $L_2$-norm. Let us consider a Gaussian  process  $X(t)$, $0\leq t\leq 1$, with zero  mean and the covariance function $G_X(t,s)=EX(t)X(s)$, $s,t\in [0,1]$. Let $\mu$ be a measure on $[0,1]$. Set
\[
\|X\|_{\mu}=\|X\|_{L_2(0,1;\mu)}=(\int\limits_0^1 X^2(t)\ \mu(dt))^{1/2}
\]
(the index  $\mu$ will be omitted if $\mu$ is the Lebesgue measure). The problem is to evaluate the asymptotics of  ${\mathbb P}\{\|X\|_{\mu}\leq\ep\}$ as $\ep \rightarrow 0$. Note that the case of 
absolutely continuous measure $\mu(dt)=\psi(t)dt$, $\psi\in L_1(0,1)$, can be easily reduced
to the case of the Lebesgue measure $\psi\equiv1$ if we replace  $X$ by the Gaussian process 
$X \sqrt{\psi}$. In general case we can assume $\mu([0,1])=1$ by rescaling. The advance of this topic
starting from well-known work \cite{S}, is reviewed in \cite{Lf} and \cite{LS}. References on
later works can be found on the site \cite{site}.

By the well-known Karhunen--Lo\'eve expansion we have the distributional equality
\begin{equation}\label{KL}
\|X\|_{\mu}^2\overset{d}{=}\sum_{j=1}^\infty \lambda_j\xi_j^2,
\end{equation}
where $\xi_j$, $j\in\mathbb N$, are  independent standard normal r.v.'s and
$\lambda_j>0$, $j\in\mathbb N$, $\sum\limits_n\lambda_n <\infty$, are the eigenvalues of 
the integral equation
\begin{equation}\label{int_eq}
\lambda y(t)=\int\limits_0^1
G_X(s,t)y(s)\mu(ds),\qquad 0\leq t\leq1.
\end{equation}
Thus, we are led to the equivalent problem of studying the asymptotic behavior as $\ep \rightarrow 0$ of  ${\mathbb P}\left\{\sum_{j=1}^\infty\lambda_j \xi_j^2\leq\ep ^2 \right\}$. 
The answer heavily depends on available information on the eigenvalues sequence $\lambda_j$. Since
the explicit formulas for these eigenvalues are known only for a limited number of processes 
(see \cite{Li}, \cite{DLL}, \cite{LS}), the study of spectral asymptotics for integral operator
(\ref{int_eq}) is of great importance.

If $G_X$ is the Green function of a boundary value problem (BVP) for ordinary differential operator 
then the sharp spectral asymptotics can be obtained by classical method traced back to Birkhoff, see \cite{Nm}. This approach developed in  \cite{NN1}, \cite{Na1}, allowed to calculate the small ball
asymptotics up to a constant for Gaussian processes of the mentioned class. Moreover, 
if eigenfunctions of (\ref{int_eq}) can be expressed via elementary or special functions then
the sharp constants can be obtained by complex variable methods, as it was done in \cite{Na2}, see 
also \cite{GHLT}--\cite{NP}, \cite{Na1}.

In a more general situation, we cannot expect to obtain the sharp asymptotics. Thus, we have to
consider only \textbf{logarithmic} asymptotics (i.e. the asymptotics  of 
$\ln {\mathbb P}\{\|X\|_{\mu}\leq \ep \}$ as $\ep\to0$). It was shown in \cite{NN2} that for this goal
it suffices, under some assumption, to know the main term of eigenvalues asymptotics
(this result was considerably generalized in recent work \cite{Na4}). This enables to apply
quite general result established in \cite{BS1}. In this way the explicit logarithmic asymptotics was
obtained for a wide class of processes including fractional Brownian motion, fractional 
Ornstein--Uhlenbeck process, the integrated versions of these processes as well as multiparameter
generalization  (for example, fractional Levi field). Thereby the absolutely continuous measures with arbitrary summable nonnegative densities were considered. In \cite{KNN} the spectral asymptotics for operators with tensor product structure were obtained. This enables to develop logarithmic asymptotics 
of $L_2$-small ball deviations for corresponding class of Gaussian fields.

The next class of problems deals with $\mu$ singular with respect to Lebesgue measure 
(it was shown in \cite{BS1} that if a measure contains absolutely continuous component then its
singular part does not influence on the main term of asymptotics). All the results here concerned
\textbf{self-similar measures}. Namely, it was shown in \cite{KL}, \cite{SV} that for $G_X$ being 
the Green function for the simplest operator $Lu\equiv-u''$, in so-called
\textbf{non-arithmetic} case the eigenvalues of (\ref{int_eq}) have the pure power asymptotics 
while in {\bf arithmetic case} the asymptotics of  $\lambda_j$ is more complicated; besides power 
term it can contain a periodic function of $\ln(j)$. This function is conjectured to be non-constant 
in all non-trivial cases, but this problem is still open. Only in simplest case of ``Cantor ladder'' 
this conjecture was proved in \cite{VSh1}, \cite{VSh2}.

The results of \cite{KL}, \cite{SV} were generalized later in two directions: in 
\cite{VSh1}--\cite{VSh3} the more general (non-sign-definite) weight functions were considered while 
in \cite{Na3} the Green functions of ordinary differential operators of arbitrary order were examined. 
The logarithmic asymptotics was obtained in \cite{Na3} for corresponding processes as well.

Finally, in the recent paper \cite{VSh4} the discrete \textbf{degenerate self-similar} weights were explored. It turns out that if $G_X$ is the Green function for the operator $Lu\equiv-u''$ then 
the eigenvalues of (\ref{int_eq}) in this case have exponential asymptotics. 
Note that method applied in preceding papers and based on renewal equation fails in the case of
degenerate self-similarity. For this reason the techniques of eigenvalues estimation
was improved in \cite{VSh4}.

In our paper we extend the result of \cite{VSh4} to the case where $G_X$ is the Green function of
a boundary problem for ordinary differential operator of arbitrary even order with the main term $(-1)^{\ell}y^{(2\ell)}$.  For simplicity we offer up the generality of weights and consider only discrete
\textbf{measure} $\mu$ with degenerate self-similarity. As a corollary, using the result of
\cite[Theorem 2]{Na4} we establish logarithmic small ball asymptotics in $L_2$-norm for corresponding 
class of Gaussian process. Let us recall that this class is rather wide, it contains in particular
${\mathfrak s}$-times integrated Brownian motion and ${\mathfrak s}$-times integrated Ornstein--Uhlenbeck process.

The paper is organized as follows. Section 2 contains auxiliary information on degenerate self-similar measures. In Section 3 the result of \cite{VSh4} is extended to the differential operators of high 
order. Then, in Section 4, we derive the logarithmic small ball asymptotics for processes of the 
class considered and give some examples. In Appendix (Section 5) a variant of the Weyl theorem used 
in the proof is given.

Let us recall some notation. A function $G(s,t)$ is called the Green function of BVP for a differential operator ${\cal L}$ if it satisfies the equation ${\cal L}G=\delta(s-t)$ in the sense of distributions and satisfies the boundary conditions. The existence of the Green function is equivalent
to the invertibility of operator ${\cal L}$  with given boundary conditions, and 
$G(s,t)$ is a kernel of the integral operator ${\cal L}^{-1}$.

$W_2^{\ell}(0,1)$ is the Hilbert space of functions $y$ having continuous derivatives up to 
$({\ell}-1)$-th order with $y^{({\ell}-1)}$ absolutely continuous on $[0,1]$ and 
$y^{({\ell})}\in L_2(0,1)$. ${\stackrel {o}{W}\!\vphantom{W}_2^{\ell}}(0,1)$ is the
subspace of functions $y\in W_2^{\ell}(0,1)$ satisfying zero boundary conditions $y(0)=y(1)=\dots=y^{({\ell-1})}(0)=y^{({\ell-1})}(1)=0$.

The principles of self-adjoint operators and  quadratic forms theory used in the paper can be found 
in the monograph \cite{BS2}.

Various constants are denoted by $c$. We point their dependence on parameters by $c(\ldots)$ if it is
necessary.

\section{Degenerate self-similar measures}

Let us recall that general concept of self-similar measure was introduced in \cite{H}. The construction of self-similar measure on interval described in \cite{SV}, see also \cite{Na3}, enables to construct 
measures with positive Hausdorff dimension of support. Let us note, that the primitive of such measure 
is always a continuous function, which is self-similar in the sense of \cite{VSh1}, \cite{VSh3}. On the
other hand, a function $f$, self-similar in the mentioned sense, need not be continuous (the criteria of
its continuity are established in \cite[Sec. 3]{Sh1}). Moreover, under some assumptions on self-similarity parameters (see below) the derivative of $f$ in the sense of distributions is a discrete measure. This measure is not self-similar in the Hutchinson sense, so we call it {\it degenerate self-similar}.\medskip

Let $0=\alpha_1<\alpha_2<\ldots <\alpha_n<\alpha_{n+1}=1$, $n\ge2$, be a partition of the segment $[0,1]$. Define quantities $a_k>0$, $k=1,\ldots,n$, by the formula $a_k=\alpha_{k+1}-\alpha_k$. Consider also a Boolean vector $(e_k)$ and (for the moment arbitrary) vectors of real numbers
$(d_k)$ and  $(\beta_k)$, $k=1,\ldots,n$.

Now we define a family of affine transformations
$$S_k(t)= a_kt+\alpha_k, \quad e_k=0;
\qquad S_k(t)=\alpha_{k+1}-a_kt, \quad e_k=1.
$$
Thus, $S_k$ moves $[0,1]$ to $[\alpha_k,\alpha_{k+1}]$ (turning it over when $e_k=1$).

\begin{dif} The affine operator ${\cal S}$ given by the formula
\begin{equation}\label{eq:auxto}
{\cal S}[f](t)=\sum\limits_{k=1}^n\left(d_k\cdot f(S_k^{-1}(t))+\beta_k\right)
\cdot\chi_{]\alpha_k,\alpha_{k+1}[}(t),
\end{equation}
(here $\chi_E$ stands for the indicator of a set $E$) is called the \textbf{similarity operator}.
\end{dif}

Thus, the graph of ${\cal S}(f)$ on the interval $]\alpha_k,\alpha_{k+1}[$ is a shifted and shrinked copy 
of the graph of $f$ on $]0,1[$.

\begin{prop}\label{st:szhim_condition} (see \cite[Lemma 2.1]{Sh1} \footnote{In
\cite{Sh1} only the transformations $S_k$ without overturn the interval were considered, but this fact doesn't influence on proof.}) Operator ${\cal S}$ is contractive in $L_\infty]0,1[$ iff
\begin{equation}\label{eq:szim1}
\max_{1\le k\le n}|d_k|<1.
\end{equation}
\end{prop}

It follows immediately from Proposition~\ref{st:szhim_condition} that under assumption (\ref{eq:szim1}) 
there exists a unique function $f\in L_\infty]0,1[$ satisfying the equation ${\cal S}(f)=f$. This 
function is called \textbf{self-similar with parameters} $(\alpha_k)$, $(e_k)$, $(d_k)$ and
$(\beta_k)$, $k=1,2,\ldots,n$.\medskip

Let us suppose now that exactly one of quantities $d_k$, $k=1,\ldots,n$, differs from zero. We denote 
by $m$ the corresponding index, $1\le m\le n$. It is obvious that in this case only $m$th element of
$(e_k)$ is relevant, and condition (\ref{eq:szim1}) becomes $|d_m|<1$.

\begin{lem}\label{tm:schet}
Under above conditions the self-similar function $f$ is piecewise constant, has bounded variation and 
possesses at most countable number of values. All discontinuity points are of the first type.
\end{lem}

\begin{proof}
Let us consider the sequence  $f_0\equiv0$, $f_j={\cal S}(f_{j-1})$. By 
Proposition~\ref{st:szhim_condition}, it converges uniformly to $f$.

It is evident, that $f_1$ is a constant on all intervals $]\alpha_k,\alpha_{k+1}[$, $k=1,\ldots,n$. Further, since only one of $d_k$s differs from zero, the function $f_2$ is piecewise constant on the interval
$]\alpha_m,\alpha_{m+1}[=S_m(]0,1[)$ and coincides with $f_1$ out of this interval. Analogously, $f_{j+1}$
is piecewise constant on the interval $S^j_m(]0,1[)$ and coincides with $f_j$ out of this interval. Moreover, the following evident equality is valid:
\begin{equation}\label{osc}
\underset{S^j_m(]0,1[)}{\mbox{Var}} f_{j+1}=d_m\cdot
\underset{S^{j-1}_m(]0,1[)}{\mbox{Var}} f_j.
\end{equation}
Thus, the limit function $f$ is piecewise constant and has finite number of values out of any interval $S^j_m(]0,1[)$, $j\in\mathbb N$. These intervals generate a sequence contracting to a point $\widehat x$, which is singular for $f$ in a sense. However, by (\ref{osc}) $f$ is continuous at $\widehat x$. The boundedness of $\underset{]0,1[}{\mbox{Var}} f$ also follows from (\ref{osc}). The proof is complete.
\end{proof}

Straightforward calculation shows that
\begin{equation}\label{eq:osob_tochka}
\widehat x=\dfrac{\alpha_{m+e_m}}{1-(-1)^{e_m}a_{m}}.
\end{equation}
In particular, (\ref{eq:osob_tochka}) implies that  $\widehat x=0$ iff  $m=1$ and $e_1=0$. Similarly,
$\widehat x=1$ iff $m=n$ and $e_n=0$.\medskip

Now, we exclude from consideration the trivial cases. Namely, we assume that $f$ has jumps at all points
$\alpha_k$, $k=2,\ldots,n$. Further, we define $f$ at discontinuity points as left-continuous function 
and define degenerately self-similar discrete signed measure $\mu$ by the formula 
$\mu([a,b])=f(b+0)-f(a)$, $0\le a\le b\le1$.

\begin{tm}\label{tm:def_string} (see also~\cite{Sh2}) The signed measure $\mu$ is a probability measure 
iff the following conditions are valid:
\begin{enumerate}
\item $d_1e_1+\beta_1=0$, $d_n(1-e_n)+\beta_n=1$;
\item $0<(-1)^{e_m}d_m<1$;
\item $\beta_k<\beta_{k+1}$, $k=1,\dots,n-1$;
\item $\beta_{m-1} <d_{m}\beta_n+\beta_{m}<\beta_{m+1}$
\end{enumerate}
(for $m=1$, only right inequality in item 4 holds; for $m=n$, only left inequality holds).
\end{tm}

\begin{proof}
Item 1 is necessary and sufficient to satisfy the equalities $f(0)=0$, $f(1)=1$, in other words,  $\mu([0,1])=1$. Further, consider the sequence $f_j$ introduced in Lemma~\ref{tm:schet}. Obviously, 
item 3 is necessary for $f_1$ to increase at discontinuity points. Condition 2 is necessary for nondecreasing of $f_2\big|_{S_m(]0,1[)}$. Next, if conditions 1-3 hold then condition 4 is necessary 
for $f_2$ to increase while crossing points $\alpha_m$ and  $\alpha_{m+1}$. Finally, items 2-4 provide 
the monotonicity of all functions $f_j$, $j\in\mathbb N$, and thus, the monotonicity of $f$.
\end{proof}

\begin{rem}
Evidently, the Hausdorff dimension of $\mu$ support is equal to zero. Therefore, 
\textbf{the spectral dimension} of $\mu$ (see \cite{F}, \cite[Sec. 5]{Na3}) is also equal to zero. 
Note that in \cite{VSh4}, \cite{Sh2} the primitive  $f$ of $\mu$ is called 
\textbf{the self-similar function of zero spectral order}.
\end{rem}

The figure illustrates the graph of function $f$ with self-similar parameters: $n=3$; $\alpha_1=0$, $\alpha_2=0.3$, $\alpha_3=0.8$,
$\alpha_4=1$; $\beta_1=0$, $\beta_2=1/3$, $\beta_3=1$; $m=2$, $d_2=1/3$, $e_2=0$. 
Formula~(\ref{eq:osob_tochka}) gives $\hat x=0.6$.\bigskip

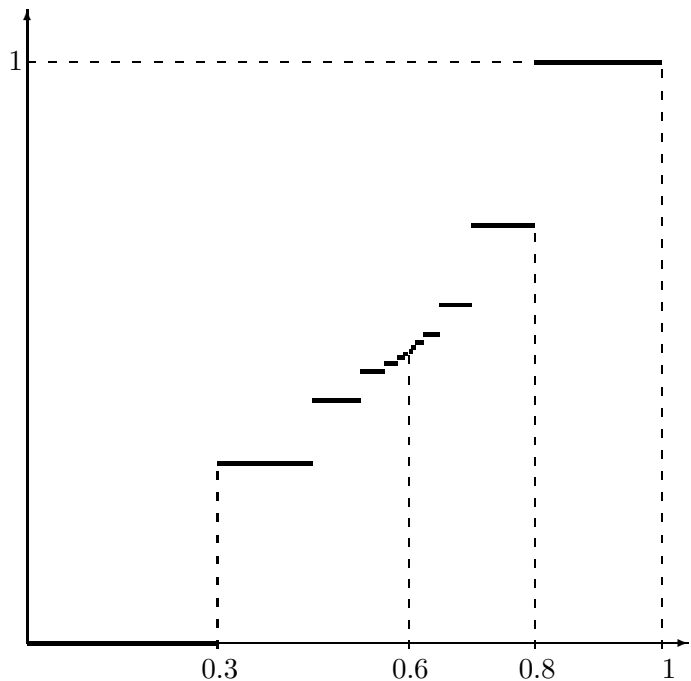
\begin{center}
\begin{picture}(300,250)
\put(0,20){\vector(1,0){250}}
\put(0,20){\vector(0,1){240}}

\linethickness{0.5mm}
\put(0,20){\line(1,0){72}}
\put(240,240){\line(-1,0){48}}

\put(72,88){\line(1,0){36}}
\put(192,178){\line(-1,0){24}}

\put(108,112){\line(1,0){18}}
\put(168,148){\line(-1,0){12}}

\put(126,123){\line(1,0){9}}
\put(156,137){\line(-1,0){6}}

\put(135,126){\line(1,0){5}}
\put(150,134){\line(-1,0){3}}

\put(140,128){\line(1,0){2.5}}
\put(147,132){\line(-1,0){1.5}}

\put(142.5,129.5){\line(1,0){1.25}}
\put(145.5,130.5){\line(-1,0){0.75}}

\thinlines
\multiput(0,240)(8,0){24}{\line(1,0){3}}
\put(-7,237){\small 1}
\put(240,7){\small 1}
\multiput(240,18)(0,8){28}{\line(0,1){3}}
\multiput(144.5,18)(0,7.7){15}{\line(0,1){3}}
\put(138,7){\small 0.6}

\multiput(72,18)(0,8){9}{\line(0,1){3}}\put(66,7){\small 0.3}
\multiput(192,18)(0,7.7){21}{\line(0,1){3}}\put(186,7){\small 0.8}

\end{picture}
\end{center}

\section{Spectral asymptotics of boundary value problems associated with degenerate self-similar measures}

Let us consider a self-adjoint, positive definite operator ${\cal L}$ generated by the differential
expression
\begin{equation}\label{operator}
{\cal L}y\equiv (-1)^{\ell}y^{(2\ell)}+\left({\cal P}_{\ell-1}y^{(\ell-1)}
\right)^{(\ell-1)}+\dots+{\cal P}_0y
\end{equation}
with suitable boundary conditions. Here ${\cal P}_i\in L_1(0,1)$, $i=0,\dots,\ell-1$.

We are interested in the eigenvalues asymptotic behavior of the BVP
\begin{equation}\label{BVP}
\lambda {\cal L}y=\mu y\quad (+\ \mbox {boundary conditions}),
\end{equation}
where $\mu$ is a probability measure constructed in Section 2.

If $G_X$ is the Green function for operator ${\cal L}$ then (\ref{BVP}) is equivalent to (\ref{int_eq}).
We denote $\lambda_j^{({\cal L}_{\mu})}$ the eigenvalues of (\ref{BVP}) enumerated in decreasing order
and repeated according to their multiplicity.

Recall (see, e.g., \cite[Sec. 10.2]{BS2}), that the counting function of eigenvalues of (\ref{BVP}) 
can be expressed in terms of quadratic form  $Q_{\cal L}$ of the operator ${\cal L}$ as follows:
\begin{equation}\label{quadr_form}
{\cal N}_{{\cal L}_{\mu}}(\lambda)\equiv\#\{j:\,\lambda_j^{({\cal L}_{\mu})}>\lambda\}=
\sup\dim\{{\mathfrak H}\subset {\cal D}(Q_{\cal L}):\ 
\lambda Q_{\cal L}(y,y)<\int\limits_0^1|y(t)|^2\mu(dt) \ \ \mbox{on}\ \ {\mathfrak H}\}.
\end{equation}

Now we can formulate the main result of this section.

\begin{tm}\label{spectral_asympt}
Given degenerate self-similar probability measure $\mu$, we have
\begin{equation}\label{count_func}
{\cal N}_{{\cal L}_{\mu}}(\lambda)\sim (n-1)\,\frac {\ln(\frac 1{\lambda})}{\ln(q)},
\qquad \lambda\to+0,
\end{equation}
where $q=\frac 1{d_m\cdot a_m^{2\ell-1}}>1$.
\end{tm}

\begin{rem} For the operator  ${\mathfrak L}y=-y''$ with the Dirichlet boundary conditions this theorem was proved in \cite{VSh4}. Moreover,
more precise result on spectrum structure of operator ${\cal L}_{\mu}$ was obtained in this case. However,
this result is not sufficient to receive sharp small ball asymptotics.
\end{rem}

\begin{proof}
First, we consider a particular case of operator ${\cal L}$, 
without lower-order terms and with the Dirichlet boundary conditions:
\begin{equation}\label{operator1}
\lambda {\mathfrak L}y\equiv\lambda (-1)^{\ell}y^{(2\ell)}=\mu y, \qquad
y(0)=y(1)=\dots=y^{({\ell-1})}(0)=y^{({\ell-1})}(1)=0.
\end{equation}

Denote by $\mathcal H$ the energy space of the operator $\mathfrak L$:
$${\cal H} ={\stackrel {o}{W}\!\vphantom{W}_2^{\ell}}(0,1);\qquad
[y,y]_{\cal H}=Q_{\mathfrak L}(y,y)=\int\limits_0^1 |y^{(\ell)}|^2.
$$
We define two subspaces in $\mathcal H$:
\begin{gather*}
\mathcal H_1:=\{y\in\mathcal H:\ y(t)\equiv 0 \text{ if }
t\in [\alpha_{m}, \alpha_{m+1}],\; y(\alpha_k)=0,\; k=2,\ldots,n\};\\
\mathcal H_2:=\{y\in\mathcal H:\ y(t)\equiv 0 \text{ if } t\not\in [\alpha_{m}, \alpha_{m+1}]\}.
\end{gather*}

Let $[\gamma_1,\gamma_2]$ be any subsegment in $]\alpha_m, \alpha_{m+1}[$ containing 
${\rm supp} (\mu)\cap\,]\alpha_m,\alpha_{m+1}[$. For instance, one can take
$$
\gamma_1=\alpha_m+a_m a_{1+e_m(n-1)};\qquad
\gamma_2=\alpha_{m+1}-a_m a_{n-e_m(n-1)}.
$$
We also need a subspace $\widehat{\mathcal H}\subset \mathcal H$ consisting of the
order $2\ell$ polynomial splines with $n+3$ interpolation points
$\alpha_k$, $k=1,\dots,n+1$, $\gamma_1$ and $\gamma_2$, satisfying the following conditions:

1. These splines vanish in $[\gamma_1,\gamma_2]$.

2. They have continuous derivatives up to the $(\ell-1)$-th order at $\alpha_m$, $\alpha_{m+1}$, 
$\gamma_1$, $\gamma_2$, and up to the $(2\ell-2)$-th order at other interpolation points. 

It is easy to see that $\dim \widehat{\mathcal H}=n-1+\Delta$ where
$$
\Delta=2(\ell-1)\quad\mbox{as}\quad m\ne 1,n;\qquad 
\Delta=\ell-1\quad\mbox{as}\quad m=1\quad\mbox{or}\quad m=n.
$$
It is also easy to check that $\mathcal H=\mathcal H_1\oplus(\mathcal H_2\dotplus\widehat{\mathcal H})$.

In turn, we decompose space $\widehat{\mathcal H}$ into the orthogonal sum of subspaces 
$\widehat{\mathcal H}=\widehat{\mathcal H}_1\oplus\widehat{\mathcal H}_2$ where
$$
\widehat{\mathcal H}_1=\{y\in\widehat{\mathcal H}: y(\alpha_k)=0,\quad k=2,\ldots,n\}.
$$
It is easily seen that
\begin{equation}\label{eq:dim_perp}
\dim\widehat{\mathcal H}_1=\Delta;\qquad
\dim\widehat{\mathcal H}_2=n-1.
\end{equation}

The quadratic form $\int_0^1|y(t)|^2\mu(dt)$ defines on $\mathcal H$ a compact self-adjoint operator
$\cal A$. Its eigenvalues certainly coincide with $\lambda_j^{({\mathfrak L}_{\mu})}$.

Denote by ${\cal B}$ and ${\cal C}$ the restrictions of ${\cal A}$ on subspaces $\mathcal H_2$ and $\widehat{\mathcal H}_2$, respectively (obviously, by construction of $\mu$ the restrictions of 
this operator on $\mathcal H_1$ and $\widehat{\mathcal H}_1$ are trivial). Then, under decomposition 
$\mathcal H=\mathcal H_1\oplus(\mathcal H_2\dotplus (\widehat{\mathcal H}_1\oplus\widehat{\mathcal
H}_2))$, the problem (\ref{BVP}) can be rewritten in matrices as follows:
\begin{equation}\label{eq:privodimost}
\lambda\begin{pmatrix}
I & 0 & 0 & 0\\
0 & I & {\cal P}_1^* & {\cal P}_2^*\\
0 & {\cal P}_1 & I & 0\\
0 & {\cal P}_2 & 0 & I
\end{pmatrix}
\begin{pmatrix}
u\\
x\\
y\\
z
\end{pmatrix}=
\begin{pmatrix}
0 & 0 & 0 & 0\\
0 & {\cal B} & 0 & 0\\
0 & 0 & 0 & 0\\
0 & 0 & 0 & {\cal C}
\end{pmatrix}
\begin{pmatrix}
u\\
x\\
y\\
z
\end{pmatrix},
\end{equation}
where $u\in \mathcal H_1$, $x\in \mathcal H_2$, $y\in \widehat{\mathcal H}_1$, 
$z\in \widehat{\mathcal H}_2$, while ${\cal P}_i$ are orthoprojectors
$\mathcal H_2\to \widehat{\mathcal H}_i$, $i=1,2$.

Formula (\ref{eq:privodimost}) shows that, to obtain asymptotics of ${\cal N}_{\cal A}(\lambda)$, we
need to consider the problem (\ref{BVP}) only in the space
$\mathcal H_2\dotplus \widehat{\mathcal H}_2$.\medskip

Let $z\in{\cal H}_2$. Setting $y(t)=z(S_m(t))\in{\cal H}$, by the homogeneity we have
$$
[z,z]_{\cal H}=a_m^{-(2\ell-1)}[y,y]_{\cal H},
$$
while the self-similarity of $\mu$ gives
$$
[{\cal B}z,z]_{\cal H}=\int_{\alpha_m}^{\alpha_{m+1}} |z(t)|^2 \mu(dt)=d_m\int_0^1
|y(t)|^2 \mu(dt)=[{\cal A}y,y]_{\cal H}.
$$
Hence (\ref{quadr_form}) implies for $\lambda>0$
\begin{equation}\label{eq:mashtab}
{\cal N}_{\cal B}(\lambda)={\cal N}_{\cal A}
(q\lambda).
\end{equation}

\begin{lem}\label{lem:transfer_function} Let $\lambda \in\mathbb{R}$ be such that the operator
${\cal C}-\lambda I$ is invertible in  $\widehat{\mathcal H}_2$. Then
\begin{equation}\label{eq:NA<transfer_func+}
{\cal N}_{\cal A}(\lambda)\ge  {\cal N}_{\widetilde{\cal B}}(\lambda)+{\cal N}_{\cal C}(\lambda),
\end{equation}
where $\widetilde{\cal B}={\cal B}-\lambda^2 {\cal P}_2^*({\cal C}-\lambda I)^{-1}{\cal P}_2$.
\end{lem}

\begin{proof}
Let $X=x+y+z$, $x\in \mathcal H_2$, $y\in \widehat{\mathcal H}_1$, $z\in \widehat{\mathcal H}_2$. 
From decomposition~\eqref{eq:privodimost}, we derive by straightforward calculation 
\begin{equation*}\label{eq:transfer_function}
[{\cal A}X,X]_{\cal H}-\lambda[X,X]_{\cal H}=
[\widetilde{\cal B}x,x]_{\cal H}-\lambda[x+y,x+y]_{\cal H}+
[{\cal C}w,w]_{\cal H}-\lambda[w,w]_{\cal H},
\end{equation*}
where $w=z-({\cal C}-\lambda I)^{-1}{\cal P}_2x$. The statement immediately follows from this relation.
\end{proof}

Now we note that for any $z\in \widehat{\mathcal H}_2$,
$$
[{\cal C}z,z]_{\cal H}=\int_0^1 |z(t)|^2\mu(dt)=\sum\limits_{k=2}^{n}\zeta_k\cdot|z(\alpha_k)|^2,
$$
where $\zeta_k=\mu(\{\alpha_k\})$. Since measure $\mu$ is assumed to be nontrivial, $\zeta_k>0$ 
for all $k=2\ldots,n$. This implies $\mbox{rank}({\cal C})=n-1$. By (\ref{eq:dim_perp}) this gives 
the invertibility of operator $\cal C$, that in turn implies
$$
{\cal N}_{\cal C}(\lambda)\equiv n-1;\qquad \|{\cal B}-\widetilde{\cal B}\|\le c\lambda^2
$$
for sufficiently small $\lambda$. In view of these formulas the inequality (\ref{eq:NA<transfer_func+}) 
provides the following relation for arbitrary $\varepsilon>0$ and $\lambda<\lambda_0(\ep)$:
$${\cal N}_{\cal A}(\lambda)\ge{\cal N}_{\cal B}(\lambda+c\lambda^2)+n-1\ge{\cal N}_{\cal B}((1+\ep)\lambda)+n-1.
$$
On the another hand, relations (\ref{eq:privodimost}) and (\ref{eq:dim_perp}) give an upper estimate:
$${\cal N}_{\cal A}(\lambda)\le  {\cal N}_{\cal B}(\lambda)+n-1+\Delta.
$$
Combining these estimates we derive, subject to (\ref{eq:mashtab}),
$${\cal N}_{\cal A}(q(1+\ep)\lambda)+n-1\le{\cal N}_{\cal A}(\lambda)\le
{\cal N}_{\cal A}(q\lambda)+n-1+\Delta,
$$
as  $\lambda<\lambda_0(\ep)$.
Iterating these inequalities we obtain two-sided estimate for  ${\cal N}_{\cal A}(\lambda)$:
\begin{equation}\label{ocenka}
(n-1)\,\frac {\ln(\frac 1{\lambda})}{\ln(q(1+\ep))}-c(\ep) \le{\cal N}_{\cal A}(\lambda)\le
(n-1+\Delta)\,\frac {\ln(\frac 1{\lambda})}{\ln(q)}+c(\ep).
\end{equation}

Now we note that the primitive of $\mu$ is a fixed point not only for the similarity operator 
${\mathcal S}$ but also for any its power. If one consider the original problem with replacing 
${\mathcal S}$ by ${\mathcal S}^M$, the problem (\ref{operator1}) doesn't change but parameters 
$q$ and $n$ replace by $q^M$ and $M(n-1)+1$, respectively. Therefore, the estimate (\ref{ocenka}) 
takes the form
$$M(n-1)\,\frac {\ln(\frac 1{\lambda})}{\ln(q^M(1+\ep))}-c(\ep) \le{\cal N}_{\cal A}(\lambda)\le
(M(n-1)+\Delta)\,\frac {\ln(\frac 1{\lambda})}{\ln(q^M)}+c(\ep),
$$
or
\begin{equation}\label{ocenka1}
(n-1)\,\frac {\ln(\frac 1{\lambda})}{\ln(q(1+\ep))}-c(\ep) \le{\cal N}_{\cal A}(\lambda)\le
(n-1+\frac {\Delta}M)\,\frac {\ln(\frac 1{\lambda})}{\ln(q)}+c(\ep).
\end{equation}
By the arbitrariness of $\ep$ and $M$ this immediately provides (\ref{count_func}).\medskip

Now we consider a general case. Integrating by parts we check that the quadratic form $Q_{\cal L}$ 
can be written as follows:
\begin{equation}\label{quadrform}
\gathered
Q_{\cal L}(y,y)=\int\limits_0^1\left[\left|y^{(\ell)}\right|^2+
\sum\limits_{i=0}^{\ell-1}{\cal P}_i\left|y^{(i)}\right|^2\right]\ dt\ +\
Q_0(y,y),\\
\stackrel{ o\ } {W_2^{\ell}}(0,1)\subset {\cal D}(Q_{\cal L})\subset
W_2^{\ell}(0,1),
\endgathered
\end{equation}
where the quadratic form $Q_0(y,y)$ contains boundary terms at the endpoints zero and one.

Consider auxiliary  quadratic form $Q_{\widetilde {\cal L}}$ with the same formal expression as 
$Q_{\cal L}$ and the same domain as $Q_{\mathfrak L}$:
$$Q_{\widetilde {\cal L}}(y,y)=Q_{\cal L}(y,y);\qquad {\cal D}
(Q_{\widetilde {\cal L}})={\cal D}(Q_{\mathfrak L})=\stackrel{ o\ }
{W_2^{\ell}}(0,1).$$

The difference of the operators ${\cal L}$ and $\widetilde {\cal L}$ is a finite-dimensional operator, 
and therefore
$${\cal N}_{{\cal L}_{\mu}}(\lambda)\sim {\cal N}_
{\widetilde {\cal L}_{\mu}}(\lambda), \qquad \lambda\to+0.$$

Further, integrating by parts we can estimate  the lower order terms in (\ref{quadrform}):
$$\left|Q_{\widetilde {\cal L}}(y,y)-Q_{\mathfrak L}(y,y)\right|
\le c\cdot\int\limits_0^1\left[\,\left|y^{(\ell-1)}y^{(\ell)}\right|+
\sum\limits_{i=0}^{\ell-1}\left|y^{(i)}\right|^2\right]\ dt.$$

This estimate shows  that $Q_{\widetilde {\cal L}}$ defines a metric which is a compact perturbation 
of the metric in $\cal H$. It was shown in the first part of the proof that the counting function 
${\cal N}_{{\cal A}}(\lambda)$ has the asymptotics (\ref{count_func}) and thus satisfies
the relation (\ref{asymptotics2}). By Lemma \ref{lem:Weyl} we obtain
$${\cal N}_{\widetilde {\cal L}_{\mu}}(\lambda)=
{\cal N}_{{\cal A}_1}(\lambda)\sim {\cal N}_{\cal A}(\lambda)=
{\cal N}_{{\mathfrak L}_{\mu}}(\lambda), \qquad \lambda\to+0,$$

\noindent and the proof is complete.
\end{proof}

\section{Small ball asymptotics. Examples}

To obtain the small ball asymptotics we use the following proposition:

\begin{prop} (\cite[Theorem 2]{Na4})
Let the counting function of the sequence  $(\lambda_j)$, $j\in\mathbb N$, has the asymptotics 
${\cal N}(\lambda)\sim\varphi(\lambda)$, as $\lambda\to +0$, where  $\varphi$ is slowly varying at zero,
i.e.
$$\lim\limits_{t\to +0}\frac {\varphi(ct)}{\varphi(t)}=1
\qquad\mbox{for any}\quad c>0.$$
Then, as $r\to +0$
\begin{equation}\label{vero}
\ln{\bf P}\left\{\sum_{j=1}^\infty\lambda_j\xi_j^2\le r
\right\} \sim -\,\frac 12 \int\limits_{\frac 1u}^1 \varphi(z)\frac {dz}z,
\end{equation}
where $u=u(r)$ is chosen satisfies
\begin{equation}\label{ur}
\frac {\varphi(\frac 1u)}{2u}\sim r, \qquad r\to +0.
\end{equation}
\end{prop}

Substituting in (\ref{ur}) $\varphi(\lambda)={\mathfrak C}\cdot\ln(\frac 1\lambda)$ we obtain
$$r\sim\frac {\mathfrak C}2 \ln(u) \quad \Longleftrightarrow \quad
u\sim\frac {{\mathfrak C}\ln(\frac 1r)} {2r}.
$$
Therefore the replacement in (\ref{vero}) $r$ by $\ep^2$ gives
\begin{equation}\label{vero1}
\ln{\bf P}\left\{\sum_{j=1}^\infty\lambda_j\xi_j^2\le\ep^2\right\} \sim
-\,\frac {{\mathfrak C}\ln^2(u)}{4}\sim
-\,{\mathfrak C}\ln^2\big(\frac 1\ep\big),\qquad \ep\to +0.
\end{equation}

As the example of formula (\ref{vero}) application, let us consider a number of well-known Gaussian
processes on $[0,1]$:\\
1) Wiener process $W(t)$;\\
2) Brownian bridge $B(t)=W(t)-tW(1)$;\\
3) centered Winer process $\overline W(t) = W(t)-\int_0^1W(s)\,ds$;\\
4) centered Brownian bridge $\overline B(t) = B(t)-\int_0^1B(s)\,ds$;\\
5) ``elongated'' Brownian bridge  $W^{(u)}(t)=W(t)-utW(1)$, $u<1$ (\cite[4.4.20]{BoS}).\\
6) generalized Slepian process  $\widehat W^{[c]}=W(t+c)-W(t)$, $c\ge1$ (\cite{Sl}).\medskip

It easy to check that the covariances of these processes are the Green functions for the operator 
${\cal L}y=-y''$ with various boundary conditions.\medskip

Processes closely related to mentioned above are\\
7) stationary Ornstein--Uhlenbeck process $U^{(\alpha)}$, $\alpha>0$;\\
8) Ornstein--Uhlenbeck process starting at zero $U^{(\alpha)}_0$, $\alpha\ne0$;\\
9) the Bogolyubov process ${\cal B}^{(\alpha)}$, $\alpha>0$ (\cite{San}, \cite{Pu2}).\medskip

The covariances of these processes,
\begin{eqnarray*}
G_{U^{(\alpha)}}(s,t)& =& \frac 1{2\alpha}\exp (-\alpha|s-t|);\\
G_{U^{(\alpha)}_0}(s,t)& =& \frac 1{2\alpha}\bigl(\exp (-\alpha|s-t|)-\exp (-\alpha(s+t))\bigr);\\
G_{{\cal B}^{(\alpha)}}(s,t)& =& \frac 1{2\alpha
}\,\frac {\exp (\alpha|s-t|)+\exp (\alpha-\alpha|s-t|)}{\exp (\alpha)-1}
\end{eqnarray*}
are the Green functions for the operator  ${\cal L}y=-y''+\alpha^2y$ with various boundary conditions.

\begin{prop}\label{ell=1} Let $\mu$ be a degenerate self-similar measure described in Section 2.
Let $X$ be one of the Gaussian processes listed in 1)-9).
Then
$$\ln {\bf P}\{||X||_{\mu}\leq\ep\}\sim -(n-1)\,\frac {\ln^2(\frac 1{\ep})}{\ln(\frac 1{d_m\cdot a_m})},
\qquad \ep\to +0.
$$
\end{prop}

\begin{proof} The statement is a consequence of Theorem \ref{spectral_asympt} (with  $\ell=1$) and 
formula (\ref{vero1}).
\end{proof}

Now we consider ${\mathfrak s}$-times integrated processes (here any $\beta_j$ equals either zero or 
one, $0\le t\le1$):
$$X_{\mathfrak s}(t)\equiv X_{\mathfrak s}^{[\beta_1,\,\ldots,\,\beta_{\mathfrak s}]}(t) =
(-1)^{\beta_1+\,\dots\,+\beta_{\mathfrak s}}\underbrace {\int\limits_{\beta_{\mathfrak s}}^t\dots
\int\limits_{\beta_1}^{t_1}}_{\mathfrak s} \ \ X (s)\ ds\ dt_1\dots\, .$$

By \cite[Theorem 2.1]{NN1}, for $X$ being one of the Gaussian processes 1)-6), the covariance of 
the process $X_{\mathfrak s}$ is the Green function for the operator 
${\cal L}y=(-1)^{{\mathfrak s}+1}y^{(2{\mathfrak s}+2)}$ with suitable boundary conditions 
(depending on endpoints of integration $\beta_j$). Analogously, for $X$  being one of the Gaussian 
processes 7)-9), the covariance of the process $X_{\mathfrak s}$ is the Green function for the 
operator ${\cal L}y=(-1)^{\mathfrak s}(-y^{(2{\mathfrak s}+2)}+\alpha^2y^{(2{\mathfrak s})})$ with 
suitable boundary conditions.

\begin{prop}\label{ell>1} Let $\mu$ be a degenerate self-similar measure described in Section 2.
Let $X$ be one of the listed Gaussian processes. Then
\begin{equation}\label{vero2}
\ln {\bf P}\{||X_{\mathfrak s}||_{\mu}\leq\ep\}\sim
-(n-1)\,\frac {\ln^2(\frac 1{\ep})}{\ln\big(\frac 1{d_m\cdot a_m^{2{\mathfrak s}+1}}\big)},
\qquad \ep\to +0.
\end{equation}
\end{prop}

\begin{proof} The statement follows from Theorem \ref{spectral_asympt} (with $\ell={\mathfrak s}+1$)
and formula (\ref{vero1}).
\end{proof}

\begin{rem} We list some more well-known Gaussian process for which Proposition \ref{ell>1} can be 
applied:\\
10) ``bridged'' (conditional) integrated Wiener process (\cite{La}, see also \cite[Proposition 5.3]{NN1})
$${\mathbb B}_{\mathfrak s}(t)=(W_{\mathfrak s}(t)\bigr|\ W_j(1)=0,\ 0\le j\le {\mathfrak s});$$
11) ${\mathfrak s}$-times centered-integrated Wiener process (see \cite[Sec. 4]{Na1}),
derived from $W(t)$ by alternate operations of centering and integration;\\
12) ${\mathfrak s}$-times centered-integrated Brownian bridge (see \cite[Sec. 3]{Na1});\\
13) the Matern process ${\cal M}^{({\mathfrak s}+1)}$ (see \cite{RW}, \cite{Pu1}) with covariance
$$G_{{\cal M}^{({\mathfrak s}+1)}}(s,t)=\frac 1{2^{2{\mathfrak s}+1}{\mathfrak s}!}\,\exp (-|s-t|)
\sum\limits_{k=0}^{\mathfrak s}\frac {({\mathfrak s}+k)!}{k!({\mathfrak s}-k)!}\,(2|s-t|)^{{\mathfrak s}-k}.
$$
\end{rem}

\section{Appendix}

The next Lemma is a variant of classical Weyl theorem (\cite{W}; see also  \cite[Lemma 1.17]{BS3}). A function $f$ is called \textit{uniformly continuous in logarithmic scale} on a set $E\subset\mathbb R_+$, 
if the function $\widetilde f=\ln\circ\, f\circ\exp$ is uniformly continuous on corresponding set.

\begin{lem}\label{lem:Weyl} Let  $\cal A$ be infinite-dimensional compact self-adjoint positive operator
in a Hilbert space $\cal H$.

{\bf 1}. Let the eigenvalues of ${\cal A}$ satisfy the relation
\begin{equation}\label{asymptotics1}
\lambda_j^{({\cal A})}\sim \psi(j),\qquad  j\to\infty,
\end{equation}
where  $\psi$ is a function uniformly continuous in logarithmic scale on $[1,+\infty[$. Then the 
asymptotics (\ref{asymptotics1}) does not change under compact perturbation of the metric in $\cal H$.

Namely, let ${\cal Q}$ be a compact self-adjoint positive operator in $\cal H$ such that 
$\min \lambda^{({\cal Q})}>-1$. Define a new scalar product in $\cal H$ by the formula 
$[u,v]_1=[u+{\cal Q}u,v]_{\cal H}$. Then
\begin{equation}\label{Weyl1}
\lambda_j^{({\cal A})}\sim \lambda_j^{({\cal A}_1)},\qquad  j\to\infty,
\end{equation}
where a positive compact operator  ${\cal A}_1$ is given by relation
\begin{equation}\label{perturb}
[{\cal A}_1u,v]_1=[{\cal A}u,v]_{\cal H}.
\end{equation}

{\bf 2}. Let the eigenvalues counting function of ${\cal A}$ satisfy the relation
\begin{equation}\label{asymptotics2}
{\cal N}_{\cal A}(\lambda)\sim \Psi(1/\lambda),\qquad  \lambda\to+0,
\end{equation}
where $\Psi$ is a function uniformly continuous in logarithmic scale on $[1,+\infty[$. Then the 
asymptotics (\ref{asymptotics2}) does not change under compact perturbations of the metric in $\cal H$, 
i.e.
\begin{equation}\label{Weyl2}
{\cal N}_{\cal A}(\lambda)\sim {\cal N}_{{\cal A}_1}(\lambda),\qquad  \lambda\to+0,
\end{equation}
where  ${\cal A}_1$ is given by (\ref{perturb}).
\end{lem}

\begin{rem} For the power-type asymptotics both statements of Lemma are equivalent.
The statement {\bf 1} works also for ``slow'' (sub-power) eigenvalues decreasing while {\bf 2} works in super-power case.
\end{rem}

\begin{rem} The second part of Lemma can be easily extracted from \cite[Theorem 3.2]{MM}. 
However, the techniques of \cite{MM} is rather complicated because a more general case of 
non-self-adjoint operators is considered. So, for the reader convenience we give
a simple variational proof of both statements.
\end{rem}

\begin{proof}
By compactness of $\cal Q$, for a given  $\delta$, we can find a finite-dimensional subspace 
${\cal H}_{\delta}$, $\dim{\cal H}_{\delta}^{\perp}=M(\delta)$, such that
$$\big|[{\cal Q}u,u]_{\cal H}\big|\le\delta[u,u]_{\cal H}, \qquad
u\in{\cal H}_{\delta}.$$ 
If $u\in{\cal H}_{\delta}$, then
$$\gathered \ 
[{\cal A}u,u]_{\cal H}<\lambda[u,u]_{\cal H}\quad\Longrightarrow\quad
[{\cal A}_1u,u]_1<\frac {\lambda}{1-\delta}\,[u,u]_1;\\
\ [{\cal A}u,u]_{\cal H}>\lambda[u,u]_{\cal H}\quad\Longrightarrow\quad
[{\cal A}_1u,u]_1>\frac {\lambda}{1+\delta}\,[u,u]_1.
\endgathered
$$

According to the variational principle, see, e.g., \cite[(1.25)--(1.26)]{BS3}, we have
\begin{equation}\label{Weyl3}
\lambda_j^{({\cal A}_1)}\ge\frac {\lambda_{j+M_{\delta}}^{({\cal A})}}
{1+\delta};\qquad\qquad {\cal N}_{{\cal A}_1}(\lambda)\le
{\cal N}_{\cal A}(\lambda(1-\delta))+M_{\delta}.
\end{equation}

Let the relation (\ref{asymptotics1}) hold. Then, dividing the first inequality in (\ref{Weyl3}) by $\psi(j)$ we obtain
$$\frac {\lambda_j^{({\cal A}_1)}}{\psi(j)}\ge\frac 1{1+\delta}\cdot\frac
{\lambda_{j+M_{\delta}}^{({\cal A})}}{\psi(j+M_{\delta})}\cdot\exp
\Big(\widetilde\psi\big(\ln(j)+\ln(1+{\textstyle\frac {M_{\delta}}j})\big)-
\widetilde\psi(\ln(j))\Big).$$
Passage to the bottom limit gives
$$\liminf\limits_{j\to\infty}\frac {\lambda_j^{({\cal A}_1)}}{\psi(j)}\ge
\frac 1{1+\delta}.$$
Changing ${\cal A}$ and ${\cal A}_1$ in (\ref{Weyl3}) and taking  $\delta\to 0$, we arrive at (\ref{Weyl1}).\medskip

Now let  (\ref{asymptotics2}) hold.  Then, dividing the second inequality in (\ref{Weyl3}) by $\Psi(1/\lambda)$ we obtain
$$\frac {{\cal N}_{{\cal A}_1}(\lambda)}{\Psi(1/\lambda)}\le
\frac {{\cal N}_{\cal A}(\lambda(1-\delta))}{\Psi(1/\lambda(1-\delta))}
\cdot \exp\Big(\widetilde\Psi\big(\ln(1/\lambda)+\ln(1/(1-\delta))\big)-
\widetilde\Psi(\ln(1/\lambda))\Big)+\frac {M_{\delta}}{\Psi(1/\lambda)}.
$$
Passage to the top limit gives
$$\limsup\limits_{\lambda\to+0} \frac {{\cal N}_{{\cal A}_1}(\lambda)}
{\Psi(1/\lambda)}\le 1+\ep,$$ 
where $\ep\to +0$ as $\delta\to +0$.

Changing ${\cal A}$ and ${\cal A}_1$ in (\ref{Weyl3}) and taking $\delta\to 0$, we arrive at 
(\ref{Weyl2}).
\end{proof}

\bigskip

We are grateful to A.A.~Vladimirov for important advice.


\end{document}